\documentclass[11pt,a4paper]{amsart}

\usepackage{epsfig} 
\usepackage{times} 
\usepackage{amsmath} 
\usepackage{amssymb}  
\usepackage{xcolor}
\usepackage[ruled]{algorithm2e}
\usepackage[noend]{algpseudocode}
\usepackage{graphicx}
\usepackage{float}
\usepackage{afterpage}
\usepackage{multicol}
\usepackage[english]{babel}
\usepackage{tikz}

\newtheorem{theorem}{Theorem}
\newtheorem{corollary}[theorem]{Corollary}
\newtheorem{lemma}[theorem]{Lemma}
\newtheorem{proposition}[theorem]{Proposition}
\newtheorem{definition}[theorem]{Definition}
\newtheorem{remark}[theorem]{Remark}

\newtheorem{assumption}[theorem]{Assumption}

\usepackage[centering]{geometry}
\usepackage{placeins}

\newcommand{\mat}[4]{\left[\begin{matrix}#1&#2\\#3&#4\end{matrix}\right]}

\newcommand{\R}{\mathbb R}
\newcommand{\C}{\mathbb C}
\newcommand{\la}{\lambda}
\newcommand{\inte}{\operatorname{int}}
\newcommand{\im}{\operatorname{im}}
\newcommand{\linspan}{\operatorname{span}}
\renewcommand{\Re}{\operatorname{Re}}
\renewcommand{\Im}{\operatorname{Im}}
\newcommand{\rank}{\operatorname{rank}}
\newcommand{\<}{\langle}
\renewcommand{\>}{\rangle}
\newcommand{\ol}{\overline}

\newcommand{\bbU}{\mathbb U}
\newcommand{\calI}{\mathcal I}
\newcommand{\calA}{\mathcal A}

\begin{document}
\title{Control of port-Hamiltonian systems with minimal energy supply}
\author[M.\ Schaller, F.\ Philipp, T.\ Faulwasser, K.\ Worthmann and B.\ Maschke]{ Manuel Schaller$^{1}$, Friedrich Philipp$^{1}$, Timm Faulwasser$^{2}$, Karl Worthmann$^{1}$ and Bernhard Maschke$^{3}$}
	\thanks{}
		\thanks{$^{1}$Technische Universit\"at Ilmemau, Institute for Mathematics, Germany
		{\tt\small \{friedrich.philipp, manuel.schaller, karl.worthmann\}@tu-ilmenau.de}.} %
	\thanks{$^{2}$TU Dortmund University, Institute of Energy Systems, Energy Efficiency and Energy Economics, Germany
		{\tt\small timm.faulwasser@ieee.org}}%
	\thanks{$^{3}$Univ Lyon, Universit{\'e} Claude Bernard Lyon 1, CNRS, LAGEPP UMR 
		5007, France {\tt\small bernhard.maschke@univ-lyon1.fr}F.\ Philipp was funded by the Carl Zeiss Foundation within the project \textit{DeepTurb---Deep Learning in und von Turbulenz}. 
		M.\ Schaller was funded by the DFG (project numbers 289034702 and 430154635). %
		K.\ Worthmann gratefully acknowledges funding by the German Research Foundation (DFG; grant WO\ 2056/6-1, project number 406141926).
		B.Maschke thanks the Institute of Mathematics of TU Ilmenau for the invitation in October 2020 which has led to this paper.}

\begin{abstract}
We investigate optimal control of linear port-Hamiltonian systems with control constraints, in which one aims to perform a state transition with minimal energy supply. Decomposing the state space into dissipative and non-dissipative (i.e. conservative) subspaces, we show that the set of reachable states is bounded w.r.t.\ the dissipative subspace. We prove that the optimal control problem exhibits the turnpike property with respect to the non-dissipative subspace, i.e., for varying initial conditions and time horizons optimal state trajectories evolve close  to the conservative subspace most of the time. We analyze the corresponding steady-state optimization problem and prove that all optimal steady states lie in the non-dissipative subspace. We conclude this paper by illustrating these results by a numerical example from mechanics. 

\end{abstract}

\maketitle

\section{INTRODUCTION}
\noindent The increasing impact of the port-Hamiltonian~(pH) framework for modelling, simulation, and analysis of interconnected physical systems 
is evidenced by the recent monographs~\cite{BrogliatoSpringer20,Jacob2012,van2014port}. 
Indeed the  pH framework extends Hamiltonian structures, which arise naturally in dynamic models of physical systems due to energy conservation and dissipation, to input and output ports. The latter point is of natural interest for control, where inputs and outputs are fundamental for feedback design. 

Actually, Hamiltonian structures with input and outputs arise in two distinct contexts in systems and control: 
(a) via energy-based modelling, where the (energy) Hamiltonian represents the total energy, 
which is \textit{the} avenue towards pH~systems, and (b) in optimal control, where the application of a variational principle leads to a Hamiltonian structure composed of the state and the adjoint/co-state/dual dynamics. In terms of (b), the (optimality) Hamiltonian is fundamental in stating Pontryagin's Maximum Principle (PMP). Moreover, in case of time-invariant Optimal Control Problems (OCPs) the  optimality Hamiltonian  is known to be invariant along optimal trajectory lifts. The classical link between both domains is given by variational modelling approaches in mechanics---i.e., the Euler-Lagrange formalism and the Hamilton formalism---which in turn can be considered as precursors of variational calculus and optimal control~\cite{Sussmann97}.


Since pH~systems are passive w.r.t.\ the usual $u^\top y$ passivity (impendance) supply rate~\cite[Chapter 7]{van2014port}, classical results on inverse optimality of passive feedbacks may be applied, see~\cite{Moylan14a,Sepulchre97a} and~\cite{Ortega08} for passivity-based feedback applied to pH~systems. Recently, the preprint~\cite{Koelsch2020} has suggested to combine inverse optimality with learning concepts. Besides these works, a few results exist on LQG~control using the structure of pH~systems \cite{Lamoline_MTNS18_LQGStochPHS,Automatica18_Wu}, on robust control 
\cite{Sato17}, and on robustness~\cite{Mehrmann20}.
In conclusion, given the common historical origins of port-Hamiltonian systems and the optimality Hamiltonian, surprisingly little has been done on exploiting pH~structures in optimal control.

Our goal in the present paper is to conduct first steps to fill this gap for linear pH~systems. In Section~\ref{sec:subspace}, after a concise analysis of the spectral properties 
and a decomposition into conservative and dissipative subspaces, we investigate the reachable set. 
In Section~\ref{sec:ocp} we consider the OCP to conduct a 
transition between given states with minimal 
supply of energy subject to input constraints. While this OCP is natural in terms of the 
objective functional, it is also singular as the energy supplied to the pH~system is given by the passivity supply rate~$u^\top y$. Subsequently, we analyze this OCP by using the Hamiltonian structure of the optimality system arising from 
the PMP in combination with the underlying pH~structure of the dynamics. 

Then, in 
Section~\ref{sec:DI}, we present our main results on the presence of turnpike phenomena in the considered class of OCPs. Turnpike properties of OCPs are a phenomenon first observed in economics; and the notion was coined in~\cite{Dorfman58}. They refer to the situation wherein, for varying initial conditions and different time horizons, the optimal solutions stay close to an optimal steady state 
during the middle part of the optimization horizon and the time spend far from the optimal steady state is bounded independent of the horizon length. We refer to \cite{Carlson91,Mckenzie76} for classical treatments, to \cite{Stieler14a,epfl:faulwasser15h,Gruene2016a,Gruene2019,Trelat15a} for recent results, and to \cite{tudo:faulwasser21b} for a recent overview.  

To the end of analysing turnpike properties of OCPs, we introduce the notion of dissipativity 
w.r.t.\ subspaces. This is related to recent results on dissipativity w.r.t.\ 
compact sets~\cite{Villanueva2020}. This way we extend recent results on dissipativity notions for OCPs~\cite{Gruene2016a,epfl:faulwasser15h}. Specifically, we show that the considered OCP is strictly dissipative w.r.t.\ 
the energy-conserving subspace under mild assumptions. This allows to establish that, for increasing horizons, the optimal solutions spend most of the time close to this conservative subspace. We also generalize the classical concept of  turnpikes being steady states---which can be understood as the attractor of infinite-horizon optimal solutions---to the turnpike being a subspace. Moreover, we show that in case of conservative pH~systems, despite the singular nature of the OCP, one can obtain optimal solutions by solving an auxiliary time-optimal problem. In other words, the technicalities of analyzing and deriving singular arcs can be avoided without loss of optimality. Finally, in Section \ref{sec:example}, we draw upon a simulation example motivated by mechanics to illustrate our findings. The paper closes with conclusions.

\section{DISSIPATIVE AND CONSERVATIVE SUBSPACES} \label{sec:subspace}
\noindent We consider (controlled) linear port-Hamiltonian systems 
\begin{subequations}\label{eq:pHsys_compl}
\begin{align}
\dot{x}(t)&=(J-R)Qx(t) + Bu(t), \qquad x(0)=x^0,\label{eq:PHsys}\\
y(t) &=  B^\top Qx(t),
\end{align}
\end{subequations}
where $J \in \mathbb{R}^{n\times n}$ is skew-symmetric, $R\in \mathbb{R}^{n\times n}$ is symmetric positive semidefinite, $Q \in \mathbb{R}^{n\times n}$ is symmetric positive definite, and $B\in \mathbb{R}^{n\times m}$ has full rank $m\le n$. In the following, (if not stated otherwise) we consider the input constraint $\mathbb U = [\underline{u}_1,\overline{u}_1]\times\dots [\underline{u}_m,\overline{u}_m]$ with $0$ in $\inte(\bbU)$, i.e., the interior of $\bbU$. We consider controls $u\in L^1(0,T;\mathbb{U})$, where $L^1(0,T;\mathbb{U})$ is the set of Lebesgue-measurable and absolutely integrable functions with values in~$\mathbb{U}$. The system \eqref{eq:PHsys} is to be understood in an almost-everywhere sense in time with solution $x\in W^{1,1}(0,T;\mathbb{R}^n)$, where $W^{1,1}(0,T;\mathbb{R}^n)$ is the space of functions $x\colon [0,T]\mapsto \mathbb{R}^n$ such that $x$ and its weak time derivative $\dot{x}$ belong to  $L^1(0,T;\mathbb{R}^n)$.

It can be easily checked that the energy Hamiltonian $H(x) \doteq \tfrac 12x^\top Qx$, which for physical systems corresponds to the total energy, 
satisfies the 
balance equation
\begin{align}
\label{eq:deriv_dissipativity}
\frac{\mathrm{d}}{\mathrm{d}t}H(x(t)) = u(t)^\top y(t) - \|R^{\frac12} Qx(t)\|_2^2.
\end{align}
Hence pH~systems of the form \eqref{eq:pHsys_compl} are passive with respect to the impendance supply rate $w = u^\top y$, see \cite[Section 6.3]{BrogliatoSpringer20}, \cite[Chapter 7]{van2014port}, and \cite{beattie2018robust} for the relation of dissipative linear time-invariant systems and port-Hamiltonian systems. Moreover note that $ u^\top y$ can be understood as the energy per time unit supplied to the system via the conjuguated (input and output) port variables.

In what follows we analyze the spectral properties of the system matrix $(J-R)Q$. Recall that, if $\la\in\C$ is a complex eigenvalue of a real matrix $A\in\R^{n\times n}$, then so is $\ol\la$ with eigenspace $\ker(A-\ol\la I_n) = \{\ol x : x\in\ker(A-\la I_n)\}$, where $\ol x = (\ol x_1,\ldots,\ol x_n)^\top$. If $\Im\la\neq 0$, from the linear independence of $x$ and $\ol x$ it follows that also $\Re(x)$ and $\Im(x)$ are linearly independent in $\R^n$. We set
$$
N_\la(A) \doteq \linspan\{\Re(x),\,\Im(x) : x\in\ker(A-\la I_n)\}\,\subset\,\R^n.
$$
This space has even dimension if $\Im\la\neq 0$. We say that a matrix $A$ is $Q$-symmetric ($Q$-skew-symmetric, $Q$-positive (semi-)definite) if it has the respective property with respect to the inner product $\<Q\,\cdot\,,\cdot\>$.
\begin{remark}[Spherical energy coordinates]\label{r:Q=I}
	Setting $\tilde x = Q^{1/2}x$, $\tilde J = Q^{1/2}JQ^{1/2}$, $\tilde R = Q^{1/2}RQ^{1/2}$, and $\tilde B = Q^{1/2}B$, the control system \eqref{eq:pHsys_compl}
	transfers into
	$$
	\dot{\tilde x} = (\tilde J-\tilde R)\tilde x + \tilde Bu,\quad y = \tilde B^\top \tilde x.
	$$
	In these coordinates, the energy becomes $H(\tilde x) = \frac12\|\tilde x\|^2$ and the matrix $Q = I_n$ and we shall use these coordinates occasionally to simplify proofs.
\end{remark}

\subsection{Spectrum and subspace decomposition for pH systems}

\noindent The following lemma provides the main result of this part. In a nutshell, there is a natural decomposition of the state space $\mathbb{R}^n$ into two subspaces such that the matrix $(J-R)Q$ is represented by a skew adjoint matrix on one subspace and by a Hurwitz matrix on the other. Hence in the sequel, we will refer to these subspaces as the conservative and the dissipative subspace.
\begin{lemma}[Spectrum and subspace decomposition]\label{l:dec}
The matrix $A \doteq (J-R)Q$ has the following spectral properties:
\begin{enumerate}
	\item[(i)]   Each eigenvalue of $A$ has non-positive real part.
	\item[(ii)]  For all $\alpha\in\R$, we have $\ker(A-i\alpha I_n) = \ker((A-i\alpha I_n)^2)$, i.e., the corresponding Jordan block is diagonal if $i\alpha\in\sigma(A)$, and it holds that
	\begin{equation}\label{e:Nia}
	N_{i\alpha}(A)\subset N_{i\alpha}(JQ)\,\cap\,\ker(RQ).
	\end{equation}
	\item[(iii)] There is a $Q$-orthogonal subspace decomposition $\R^n = M_1\oplus_Q M_2$ with respect to which
	\begin{equation}\label{e:mat_dec}
	JQ = \mat{J_1}00{J_2}\quad\text{and}\quad RQ = \mat 000{R_2},
	\end{equation}
	such that $J_1$ and $J_2$ are $Q$-skew-symmetric (on $M_1$ and $M_2$, respectively), $R_2$ is $Q$-positive semidefinite, and $J_2-R_2$ is Hurwitz, i.e., all eigenvalues have negative real part.
	\end{enumerate}
\end{lemma}
\begin{proof}
In view of Remark~\ref{r:Q=I} we may assume WLOG that $Q=I$.

(i). Let $(J-R)x = \la x$ for some $\la\in\C$ and $x\in\C^n$, $\|x\|=1$. Then $\la = \<\la x,x\> = \<Jx,x\> - \<Rx,x\>$. Since $J$ is skew-symmetric, we have $\Re\<Jx,x\> = 0$ and thus $\Re\la\le 0$.

(ii). Let $\la = i\alpha$, $\alpha\in\R$, and assume that $(J-R)x = i\alpha x$. Then, by the same calculation as before, $Rx = 0$ and thus $Jx = i\alpha x$. This proves the inclusion \eqref{e:Nia}. Let $y\in\C^n$ such that $(J-R-i\alpha)y = x$. Then $\|x\|^2 = \<x,(J-R-i\alpha)y\> = \<(-J-R+i\alpha)x,y\> = 0$, so $x=0$ and hence $\ker(A-i\alpha I_n) = \ker((A-i\alpha I_n)^2)$.

(iii). Define $M_1 \doteq \linspan\{N_{i\alpha} : i\alpha\in\sigma(A)\}$. By \eqref{e:Nia}, this space is both $J$- and $R$-invariant. Hence, so is $M_2 \doteq M_1^\perp$. Since $R$ vanishes on $M_1$, it is clear that the representations of $J$ and $R$ with respect to the decomposition $\R^n = M_1\oplus M_2$ take the form \eqref{e:mat_dec}. Since $N_{i\alpha}(J_2-R_2)\subset N_{i\alpha}(J-R)\subset M_1$, and thus $N_{i\alpha}(J_2-R_2)=\{0\}$, it is also clear that $J_2-R_2$ is Hurwitz.
\end{proof}

\begin{remark}
	Note that $R_2$ might still have a non-trivial kernel.
\end{remark}

Now, with respect to the decomposition $\R^n = M_1\oplus_Q M_2$ from Lemma \ref{l:dec} the control system \eqref{eq:PHsys} takes the form
\begin{subequations}
	\label{eq:PHsys_dec}
\begin{align}
\label{eq:PHsys_dec_1}
\dot x_1 &= J_1x_1 + B_1u \qquad\qquad\quad\; x_1(0) = x_1^0\\
\label{eq:PHsys_dec_2}
\dot x_2 &= (J_2-R_2)x_2 + B_2u \qquad x_2(0) = x_2^0.
\end{align}
\end{subequations}
This decomposes the system into a conservative~
\eqref{eq:PHsys_dec_1} and a dissipative subsystem~\eqref{eq:PHsys_dec_2}. The conservative subspace $M_1$ is contained in the null space of $RQ$ which will play a specific role in the optimal control problem we consider below.
\subsection{Reachibility sets}
\noindent In this part we will briefly discuss the reachibility sets in view of the decomposition of \eqref{eq:PHsys_dec}.

\begin{lemma}[Description of the reachable set]\label{l:reachable}
Assume that system~\eqref{eq:PHsys} is controllable, i.e., $\rank(B,AB,\ldots,A^{n-1}B) = n$ for $A = (J-R)Q$, and that $\mathbb{U}$ is compact. Then the following statements hold:
\begin{enumerate}
	\item[(i)] For every state $x_1^\star \in M_1$ there exist a time $T>0$ and a control $u\in L^1(0,T;\mathbb{U})$ which steers $x_1^{0}$ to $x_1^\star$ at time $T$ under the dynamics in \eqref{eq:PHsys_dec_1}.
	\item[(ii)] The set of states in $M_2$ that can be reached from $x_2^{(0)}$ in arbitrary time under the dynamics \eqref{eq:PHsys_dec_2} is bounded in $M_2$.
\end{enumerate}
\end{lemma}
\begin{proof}
(i). Since $\sigma(J_1)\subset i\,\R$, by \cite[Theorem 5, p. 45]{Macki2012} there exist $T_1\ge 0$ and a control $u\in L^1(0,T_1;\mathbb{U})$ that steers $x_1^0$ into $0\in\R^n$ at time $T_1$. Let $x_1\in W^{1,1}(0,T_1;\mathbb{R}^n)$ denote the corresponding internal state. By the same reason there exist a time $T_2$ and a control $v\in L^1(0,T_2;\mathbb{U})$, which steers $x_1^\star $ to $0$ in time $T_2$ under the dynamics
$$
\dot z_1(t) = -J_1z_1(t) - B_1v(t).
$$
By $z_1$ denote the corresponding state solution. Set $T \doteq T_1+T_2$ and define $x_1(t) \doteq z_1(T-t)$ as well as $u(t) \doteq v(T-t)$, $t\in (T_1,T]$. Then $u\in L^1(0,T;\mathbb{U})$, $x_1$ is absolutely continuous on $[0,T]$, and
\begin{align*}
\dot x_1(t)
&= -\dot z_1(T-t) = J_1z_1(T-t) + B_1v(T-t)\\
&= J_1x(t) + B_1u(t)
\end{align*}
for $t\in (T_1,T]$. Also, $x_1(0) = x_1^0$ and $x_1(T) = z_1(0) = x_1^\star $.

(ii). This can be easily seen from the variation of constants formula. Indeed, for any control $u\in L^1(0,T;\mathbb{U})$ the solution of \eqref{eq:PHsys_dec_2} can be represented as
$$
x_2(t) = e^{tA_2}x_2^0 + \int_0^t  e^{(t-s)A_2}B_2u(s)\,ds,
$$
where $A_2 = J_2-R_2$. As $A_2$ is Hurwitz, there exists $\mu >0$, $M\geq 1$ such that $\|e^{tA_2}\|\le Me^{-\mu t}$. Hence,
\begin{align*}
\|x(t)\|
&\le Me^{-\mu t}\|x_2^0\| + \int_0^t  Me^{-\mu(t-s)}\|B_2\|\|u(s)\|\,ds\\
&\le M\|x_2^0\| + \frac{M}{|\mu|}\|B_2\|\Big(\max_{v\in \mathbb{U}}\|v\|\Big)
\end{align*}
for all times $t\ge 0$.
\end{proof}

\begin{corollary}
Under the assumptions of  Lemma \ref{l:reachable} the reachable states of the system \eqref{eq:PHsys} are contained in $M_1\oplus_Q K$, where $K\subset M_2$ is compact in $M_2$.
\end{corollary}

\section{OCP: MINIMUM ENERGY SUPPLY}\label{sec:ocp}
\noindent Having discussed control-theoretic properties of pH systems in the previous section, we now introduce the considered optimal control problem. In pH~systems, the energy supplied to the system is given by $u(t)^\top y(t)$, cf. the energy balance \eqref{eq:deriv_dissipativity}. This induces a very natural optimization objective  when performing a state transition, i.e., trying to find a control that steers the state from an initial value~$x_0$ to a target $x_T$.
\noindent Hence, we turn to the OCP
\begin{align}\label{OCP:terminal}
\begin{split}
\min_{u\in L^1(0,T;\mathbb{U})}& C(u)\doteq\int_0^T    u(t)^\top  y(t) \,\mathrm{d}t\\
\text{s.t.}\quad \dot{x}(t) &= (J-R)Qx(t) + Bu(t),\\
x(0)&=x^0,\quad x(T)=x_T,\\
y(t) &= B^\top Qx(t).
\end{split}
\end{align}
Note that using the energy balance equation \eqref{eq:deriv_dissipativity} the cost functional of \eqref{OCP:terminal} may be expressed as
\begin{align}\label{eq:dissipativity}
	C(u) = H(x(T))-H(x(0))  + \int_0^T  \|R^{\frac12}Qx(t)\|_2^2\,\mathrm{d}t.
\end{align}

As is well known, the task of steering $x^0$ to $x_T$ at (any) time $T$ is surely feasible in the case where $\mathbb U = \R^m$ and $((J-R)Q,B)$ is controllable. However, as Lemma \ref{l:reachable} shows, this is much more delicate if $\mathbb U$ is compact. We make the following assumption to ensure feasibility of the OCP~\eqref{OCP:terminal}.
\begin{assumption}\label{as:feasible}
There exists a control $u\in L^1(0,T;\bbU)$ which steers $x^0$ to $x_T$ at time $T$ under the dynamics in \eqref{eq:PHsys}.
\end{assumption}

Then the next proposition follows immediately from \cite[Theorem 2, p.\ 91]{Macki2012}.

\begin{proposition}[Existence of optimal solutions]\label{prop:ExistenceOptCon}
Under Assumption \ref{as:feasible} the OCP \eqref{OCP:terminal} has an optimal solution.
\end{proposition}

\subsection{Necessary Optimality Conditions and Singular Arcs}

\noindent We deduce the first-order optimality conditions 
for OCP~\eqref{OCP:terminal}. By the PMP (see, e.g., \cite{Liberzon12}), for any optimal solution $(x^\star ,u^\star)$ of~\eqref{OCP:terminal} there is $(\lambda_0,\lambda)\in\mathbb{R}_{\geq 0}\times W^{1,1}(0,T;\mathbb{R}^n)$, $(\lambda_0,\lambda(t))\neq 0$ for all $t\in [0,T]$, such that 
\begin{align}\label{eq:OC_2}
\begin{split}
	\dot{x}^\star(t) &= \phantom{-}(J-R)Qx^\star(t) + Bu^\star(t)\\
	\dot{\lambda}(t)&= -\la_0QBu^\star(t) + Q(J+R)\lambda(t)\\
	u^\star(t) &\in \arg\min_{\tilde{u}\in \mathbb{U}}\,\tilde u^\top B^\top (\la_0Qx^\star(t)+\la(t)).
\end{split}
\end{align}
Due to the fact that the optimality Hamiltonian is affine linear in $u$, i.e.,\ the OCP is singular, we consider the $i$-th \textit{switching function}~$s_i$, $i \in\{1,\dots, m\}$, defined by
$$
s_i(t) \doteq \big(B^\top(\la_0Qx(t)+\la(t))\big)_i = b_i^\top(\la_0Qx(t)+\la(t)),
$$
where the $b_i$ denote the columns of the matrix $B$. Since $\tilde u^\top B^\top (\la_0Qx(t)+\la(t)) = \sum_i\tilde u_is_i(t)$ in~\eqref{eq:OC_2}, it follows that $u_i^\star = \underline{u}_i$ on the open set $\{t : s_i(t)>0\}$ and $u_i^\star = \overline{u}_i$ on $\{t:s_i(t)<0\}$. However, if
$$
Z_i\doteq\{t : s_i(t)=0\}
$$
has positive measure, the OCP is said to exhibit a {\em singular arc}~\cite{Liberzon12} and it is well understood that the presence of singular arcs complicates the analysis of OCPs, cf. the classical example of Fuller~\cite{Fuller60a}, see also~\cite{Liberzon12}. Here, however, we completely characterize the optimal control in dependence of the optimal state trajectory and the corresponding adjoint on such singular arcs under a certain structural assumption. 

\begin{theorem}[Singular controls]\label{t:sing}
Assume that $\la_0 > 0$ and $\im(B)\cap\ker(RQ) = \{0\}$ hold. If $(x,u,\la,\la_0)$ satisfies the optimality system~\eqref{eq:OC_2} of OCP~\eqref{OCP:terminal}, then $u$ is completely determined by $x$ and $\la$. 

Specifically, given a subset $\calI\subset\{1,\ldots,m\}$, set $\calA = \calI^c$, $u_\calI = (u_i)_{i\in\calI}^\top$, $u_\calA = (u_i)_{i\in\calA}^\top$, $B_\calI = (b_i)_{i\in\calI}$, and $B_\calA = (b_i)_{i\in\calA}$. Then we have
$$
u_\calI = M^{-1}B_\calI^\top\big[\tfrac 12(QA^2x + (A^2)^\top\la) - QRQB_\calA u_\calA\big],
$$
on $\bigcap_{i\in\calI}Z_i\setminus\bigcup_{i\in\calA}Z_j$, where $M = B_\calI^\top QRQB_\calI$ and $A = (J-R)Q$.
\end{theorem}
\begin{proof}
Let $s_\calI = B_\calI^\top(\la_0Qx+\la)$ and $Z = \bigcap_{i\in\calI}Z_i\setminus\bigcup_{i\in\calA}Z_j$. It is easy to see that $\dot s_\calI = 0$ a.e.\ on $Z$. Since $\dot s_\calI = B_\calI^\top(\la_0Q\dot x+\dot\la) = B_\calI^\top(\la_0QAx-A^\top\la)$ is absolutely continuous, it follows that also $\ddot s_\calI=0$ a.e.\ on $Z$. Therefore, setting $v = \frac 12(\la_0QA^2x + (A^2)^\top\la)$, we have
\begin{align*}
0
&= \ddot s_\calI = B_\calI^\top(\la_0QA\dot x-A^\top\dot\la)\\
&= B_\calI^\top\big[\la_0QA(Ax+Bu) - A^\top(-\la_0QBu-A^\top\la)\big]\\
&= B_\calI^\top\big[2v + \la_0(QA+ A^\top Q)Bu\big]\\
&= 2B_\calI^\top v - 2\la_0B_\calI^\top QRQBu,
\end{align*}
where we have used that $QA+ A^\top Q = -2QRQ$. Thus,
\begin{align*}
B_\calI^\top QRQB_\calI u_\calI = B_\calI^\top v - B_\calI^\top QRQB_\calA u_\calA,
\end{align*}
which proves the theorem. The matrix $B_\calI^\top QRQB_\calI$ is positive definite (and thus indeed invertible) since $\ker(B_\calI^\top QRQB_\calI) = \ker(RQB_\calI)$. So if $RQB_\calI z = 0$, then $B_\calI z\in\im(B)\cap\ker(RQ) = \{0\}$ and hence $B_\calI z=0$. As $B$ has full rank, we conclude $z=0$.
\end{proof}

In the following theorem we show that, if system~\eqref{eq:PHsys} is normal, $\lambda_0=0$ implies that there are no singular arcs.
\begin{theorem}
If $\la_0=0$ and system \eqref{eq:PHsys} is normal, i.e., the matrix $(b_i,Ab_i,\ldots,A^{n-1}b_i)$ is invertible for every column $b_i$ of $B$, then OCP~\eqref{OCP:terminal} does not exhibit singular arcs.
\end{theorem}
\begin{proof}
Since $\la_0=0$, the $k$-th derivative of $s_i$ is given by $s_i^{(k)} = (-1)^kb_i^\top(A^k)^\top\la$ for $k=0,1,\ldots$. If $s_i=0$ on a set of positive measure $Z$, then also $s_i^{(k)}=0$ a.e.\ on $Z$ for $k=1,\ldots,n-1$. Since \eqref{eq:PHsys} is normal, it follows that $\la=0$ on $Z$. But this contradicts the condition that $(\la_0,\la(t))\neq 0$ for all $t\in [0,T]$.
\end{proof}

\subsection{The Lossless Case $R=0$}
\noindent If the considered pH~system is lossless (or conservative), the computation of optimal solutions can be simplified.
To this end, consider the free end-time counterpart to OCP \eqref{OCP:terminal}:
\begin{align}
\label{OCP:terminal_T}
\begin{split}
\min_{T\geq 0,u\in L^1(0,T;\mathbb{U})}& C(u)\doteq\int_0^T    u(t)^\top  y(t) \,\mathrm{d}t\\
\text{s.t.}\quad\dot{x}(t) &= JQx(t) + Bu(t),\\
x(0)&=x^0,\quad x(T)=x_T,\\
y(t) &= B^\top Qx(t).
\end{split}
\end{align}

\begin{lemma}[Feasibility implies optimality] %
Let the pH system~\eqref{eq:pHsys_compl} be controllable and lossless ($R=0$). Then, any feasible solution $u\in L^1(0,T;\mathbb{U})$ of the OCP~\eqref{OCP:terminal} is optimal. %
Moreover, any feasible solution to the free end-time problem  \eqref{OCP:terminal_T} delivers the same performance as optimal solutions in  OCP \eqref{OCP:terminal} (provided they exist).
\end{lemma}
\begin{proof}
The assertion follows from \eqref{eq:dissipativity} as for $R =0$ the value of the objective functional $C(u)$ is completely determined by the initial condition $x^0$ and the terminal condition $x_T$. Hence, also the free end-time problem \eqref{OCP:terminal_T} delivers the same performance. 
\end{proof}

This observation motivates to obtain an optimal solution of OCP~\eqref{OCP:terminal} via an auxiliary problem, for which the analytic solution is known. Indeed, by the bang-bang principle, if there is a control that steers $x^0$ to $x_T$, then there is a bang-bang control that does so as well, cf.\ \cite[Theorem 10, p.48]{Macki2012}.
Consequently, we seek a {\em time-optimal} solution, i.e., we solve
\begin{align}
\begin{split}\label{OCP:timeopt}
&\min_{T\geq 0,u\in L^1(0,T;\mathbb{U})}\;T\\
&\text{s.t.}\quad \dot{x}(t)=JQx(t)+Bu(t),\\
&x(0)=x^0,\quad  x(T)=x_T.
\end{split}
\end{align}
Since $Q$ is symmetric positive definite and $J$ is skew-symmetric for the considered pH~systems, $\sigma(JQ) = \sigma(Q^{1/2}JQ^{1/2})\subset i\,\mathbb{R}$ holds. The following lemma can be proved analogously to Lemma \ref{l:reachable} (i).

\begin{lemma}\label{l:cool}
Let \eqref{eq:PHsys} (with $R=0$) be controllable. Then for any $x^0,x_T\in\R^n$ there exist a time $T\ge 0$ and a control $u\in L^1(0,T;\mathbb{U})$ which steers $x^0$ to $x_T$ in time $T$.
\end{lemma}

Hence we arrive at the main insight of this section: under the assumption of lossless-ness, optimal solutions for OCP~\eqref{OCP:terminal_T} can be obtained solving \eqref{OCP:timeopt}. Moreover, the performance of these solutions evaluated in the objective \eqref{OCP:terminal} is identical to solving OCP \eqref{OCP:terminal} directly.

\section{DISSIPATIVITY, TURNPIKE AND STEADY STATES} 
\label{sec:DI}
\noindent We analyse the OCP~\eqref{OCP:terminal} in a dissipativity framework for the general (dissipative) case $R \not = 0$. %
Indeed beginning with \cite{Angeli12a} there has been widespread interest in dissipativity notions of OCPs in context of model predictive control, see \cite{Stieler14a,epfl:faulwasser15h,Gruene2019a}. The driving force behind these investigations is the close relation between dissipativity and turnpike properties of OCPs~\cite{epfl:faulwasser15h,Gruene2016a,Trelat2020}. 

\subsection{Equilibria of the extremal dynamics}

\noindent Consider the steady-state problem corresponding  to OCP \eqref{OCP:terminal}, i.e.,
\begin{align}\label{OCP:steady}
\begin{split}
&\min_{\hat{u}\in \mathbb{U}} \hat{u}^\top \hat{y}\\
\text{s.t.}\quad0 &= (J-R)Q\hat{x} + B\hat{u},\\
\hat{y} &= B^\top Q\hat{x}.
\end{split}
\end{align}
The first-order necessary conditions, cf. \cite{Troeltzsch2010}, are the optimality system of \eqref{eq:OC_2} considered at steady state, i.e., if $(\hat{x},\hat{u})$ solves \eqref{OCP:steady}, there exists a Lagrange multiplier $\hat{\lambda}\in \mathbb{R}^n$ such that 
\begin{align}
\begin{split}\label{eq:steady_sys}
0&= (J-R)Q\hat{x} + B\hat{u}\\
0 &= - B\hat{u} + (J+R)\hat{\lambda}\\
\hat{u}^\top \big(B^\top (Q\hat{x}+\hat{\lambda})\big)&\leq \tilde{u}^\top \big(B^\top (Q\hat{x}+\hat{\lambda})\big) \quad \forall \tilde{u}\in \mathbb{U}.
\end{split}
\end{align}

\begin{theorem}\label{t:oss}
If $(\hat x,\hat u,\hat\la)$ is an optimal steady state, then 
\begin{align*}
B^\top (Q\hat x+\hat\la) = 0
\quad\text{and}\quad
J(Q\hat x+\hat\la) = R\hat\la = RQ\hat x = 0.
\end{align*}
In particular, if $J-R$ is invertible or $((J-R)Q,B)$ is controllable, then $Q\hat x + \hat\la = 0$ holds.
\end{theorem}
\begin{proof}
Set $A = (J-R)Q$ and let $(\bar x,\bar u,\bar\la)$ be a solution of \eqref{eq:steady_sys}. Then
\begin{align*}
\<R\bar\la,\bar\la\>
&= \<(J+R)\bar\la,\bar\la\> = \<B\bar u,\bar\la\> = -\<A\bar x,\bar\la\>\\
\<RQ\bar x,Q\bar x\>&= \<(R-J)Q\bar x,Q\bar x\> = \<B\bar u,Q\bar x\>\\
&= \<Q(J+R)\bar\la,\bar x\> = -\<A\bar x,\bar\la\>.
\end{align*}
This shows $\|R^{1/2}\bar\la\|^2 = \|R^{1/2}Q\bar x\|^2 = -\<A\bar x,\bar\la\>$. But on our way we also saw that
$$
-\<A\bar x,\bar\la\> = \tfrac 12\big(\<B\bar u,\bar\la\> + \<B\bar u,Q\bar x\>\big) = \tfrac 12\<B\bar u,Q\bar x+\bar\la\>.
$$
Thus, $\bar u^\top B^\top(Q\bar x + \bar\la) = \|R^{1/2}Q\bar x\|^2 = \|R^{1/2}\bar\la\|^2\ge 0$ for each admissible $(\bar x,\bar u,\bar\la)$. For the optimal triple $(\hat x,\hat u,\hat\la)$ we conclude that
$$
0 = \hat{u}^\top \big(B^\top (Q\hat{x}+\hat{\lambda})\big)\le\tilde{u}^\top \big(B^\top (Q\hat{x}+\hat{\lambda})\big)
$$
for all $\tilde{u}\in \mathbb{U}$ and therefore $B^\top (Q\hat{x}+\hat{\lambda}) = 0$. Finally, we conclude from \eqref{OCP:steady} and \eqref{eq:steady_sys} that $J(Q\hat x + \hat\la) = 0$.

As to the ``in particular''-part, let $\hat z = Q\hat x + \hat\la$ and note that $B^\top\hat z = 0$ and $A^\top\hat z = 0$. If $J-R$ is invertible, then so is $A$ and $\hat z=0$ follows immediately. Moreover, $\hat z$ is contained in the null space of the transpose of the Kalman matrix $(B,AB,\dots,A^{n-1}B)$ and hence vanishes if $(A,B)$ is controllable.
\end{proof}

\subsection{Strict dissipativity and the turnpike property}
\noindent We first present a lemma that relates the dissipation term in the right-hand side of $\eqref{eq:deriv_dissipativity}$ to the distance to the kernel of $R^{\frac12}Q$ denoted by $\operatorname{dist}(x,\ker(R^{\frac12}Q))\doteq \inf_{v\in \ker(R^{1/2}Q)}\|v-x\|$.

\begin{lemma}\label{l:dist}
There are constants $c_1,c_2 > 0$ such that
\begin{align*}
c_1\operatorname{dist}(x,\ker R^{\frac12}Q)\leq \|R^{\frac12}Qx\|\leq c_2 \operatorname{dist}(x,\ker R^{\frac12}Q)
\end{align*}
for all $x\in \mathbb{R}^n$.
\end{lemma}
\begin{proof}
Let $D \doteq QRQ$. We have $\ker D = \ker(R^{1/2}Q)$. Now, decompose $\R^n = \ker D\oplus\im D$ and write $D = 0\oplus D_2$, where $D_2$ is positive definite. If $x\in\R^n$, also decompose $x = x_1\oplus x_2$ accordingly. Then $\|R^{\frac 12}Qx\|^2 = \<Dx,x\> = \<D_2x_2,x_2\>$. This implies $\la_{\min}^{1/2}\|x_2\|\le\|R^{\frac 12}Qx\|\le\la_{\max}^{1/2}\|x_2\|$, where $\la_{\min}$ ($\la_{\max}$) is the smallest (largest, resp.) eigenvalue of $D_2$. The claim now follows from $\|x_2\| = \operatorname{dist}(x,\ker D)$.
\end{proof}

Next we present a generalization of the notion of strict dissipativity for OCPs which has appeared first in a discrete-time context in \cite{Angeli12a}. Here, we introduce a novel notion that formulates dissipativity with respect to a subspace. Similar ideas have been developed in \cite{Villanueva2020}, where the authors consider dissipativity with respect to a compact set.

$\mathcal K$ denotes the set of continuous and strictly increasing functions $\alpha$ from $[0,\infty)$ into itself with $\alpha(0) = 0$.

\begin{definition}[Dissipativity with respect to subspaces]\label{def:DI_V}
Let $\ell :\R^n\times \mathbb{R}^m\to\R$ be a running cost, $A\in\R^{n\times n}$, and $B\in\R^{n\times m}$. An OCP of the form
\begin{align}\label{OCP:linear}
\begin{split}
\min_{u\in L^1(0,T;\mathbb{U})} &\int_0^T \ell(x(t),u(t))\,dt\\
\text{s.t.}\quad\dot{x}(t) &= Ax(t) + Bu(t),\\
x(0) &= x^0,\quad x(t) = x_T
\end{split}
\end{align}
is said to be {\em strictly dissipative} with respect to a subspace $\mathcal{V}\subset\mathbb{R}^n$ if there exists a storage function $S:\mathbb{R}^n\to[0,\infty)$ and a function $\alpha\in\mathcal{K}$ such that all optimal controls $u^\star\in L^1(0,T;\mathbb{U})$ of \eqref{OCP:linear} and associated states $x^\star\in W^{1,1}(0,T;\mathbb{R}^n)$ satisfy the dissipation inequality
\begin{align}\label{eq:sDI}
\begin{split}
S(x_T)- S(x^0) \leq  \int_0^T\ell(x^\star(t),u^\star(t))-\alpha(\operatorname{dist}(x^\star(t),\mathcal{V}))\,\mathrm{d}t.
\end{split}
\end{align}
\end{definition}

If $\mathcal{V}=\{0\}$, the above definition coincides with the usual definition of strict dissipativity, cf.\ \cite{epfl:faulwasser15h,Gruene2019a}.

An immediate consequence of this definition 
is the following turnpike result stating that on a large portion of the horizon $[0,T]$, the optimal trajectories of \eqref{OCP:linear} reside close to the subspace $\mathcal{V}$. 
If the system is strictly dissipative with respect to a steady state $\bar{x}\in \mathbb{R}^n$, i.e., setting $\mathcal{V}=\{\bar{x}\}$ in Definition~\ref{def:DI_V}, one can conclude a turnpike property respect to this steady state, cf.\ the recent overview article \cite{tudo:faulwasser21b}. The next lemma provides an extension of this concept to subspaces.

\begin{lemma}[Str. dissipativity implies subspace turnpike]\label{l:turnpike}
Denote by $x(\cdot,x^0,u)$ the solution of the ODE in \eqref{OCP:linear} with initial value $x^0$ and control $u\in L_1(0,T;\mathbb{U})$. Assume that the OCP \eqref{OCP:linear} is strictly dissipative with respect to a subspace $\mathcal{V}\subset\mathbb{R}^n$ and that
\begin{enumerate}
\item[(a)] there are $T_1\geq 0$ and $u_1 \in L_1(0,T_1;\mathbb{U})$ such that $x(T_1,x^0,u_1)=0$.
\item[(b)] there are $T_2\geq 0$ and $u_2 \in L_1(0,T_2;\mathbb{U})$ such that $x(T_2,0,u_2)=x_T$.
\end{enumerate}
Then, for all compact sets $K\subset\mathbb{R}^n$ and $\varepsilon> 0$ there is $C_{K,\varepsilon}>0$ independent of $T$ such that for all optimal solutions $x^\star (t)$ starting in $K$,
\begin{align}
\label{eq:turnpike_V}
\mu[t\in [0,T] :\operatorname{dist}(x^\star (t),\mathcal{V}) \geq \varepsilon]< C_{K,\varepsilon},
\end{align}
where $\mu$ denotes the standard Lebesgue measure on $\mathbb{R}$.
\end{lemma}
\begin{proof}
The claim follows by straightforward adaptation of the proofs in \cite[Section 3]{Gruene2019a} or \cite[Theorem 2]{epfl:faulwasser15h}.
\end{proof}

\begin{theorem}[Strict dissipativity of  OCP~\eqref{OCP:terminal}]\label{t:ph_turnpike}
OCP \eqref{OCP:terminal} is strictly dissipative with respect to $\ker(R^{\frac 12}Q)$. Moreover, if \eqref{eq:PHsys} is controllable and $x_T$ is reachable from $0$ under the dynamics in \eqref{eq:PHsys}, optimal trajectories of \eqref{OCP:terminal} exhibit a subspace turnpike behaviour as described in Lemma \ref{l:turnpike}.
\end{theorem}
\begin{proof}
The first claim follows immediately from \eqref{eq:dissipativity} and Lemma \ref{l:dist}. And since condition (a) in Lemma \ref{l:turnpike} is satisfied in our case (see \cite[Theorem 5, p.\ 45]{Macki2012}), the second claim is a direct consequence of Lemma \ref{l:turnpike}.
\end{proof}

\begin{remark}[Available storage]\label{rem:availStor}
In the foundational work of Jan Willems (cf.\ \cite{Willems1972a}) the available storage for a dissipative system with supply rate $w:\mathbb{R}^n\times\mathbb{R}^m\to\mathbb{R}$ is defined by
		\begin{align*}
\tilde{S}(x^0)\doteq -\sup_{T\geq 0,u\in L^1(0,T;\mathbb{U})}  \int_0^Tw((x(t,u,x^0),u(t)))\,\mathrm{d}t.
		\end{align*}
It is well-known that boundedness of the available storage is a necessary and sufficient condition for dissipativity. 
Considering the particular supply rate $w(x,u) = u^\top B^\top Qx-\|R^{\frac12} Qx\|^2$
we get for solutions of \eqref{eq:PHsys} that
		\begin{align*}
\tilde{S}(x^0) = \sup_{T\geq 0,u\in L^1(0,T;\mathbb{U})} &(H(x(0))- H(x(T,u,x^0)). 
		\end{align*}
For a in-depth treatment of passivity inequalities for pH~systems the interested reader is also referred to \cite{van2008balancing}.
\end{remark}
	
The next result summarizes the main insights.

\begin{theorem}[Dissipativity with respect to subspaces]
	\label{thm:dissipativity}
Assume that $J-R$ is invertible. Then the following hold:
\begin{enumerate}
\item[(i)] Optimal steady states and corresponding Lagrange multipliers satisfy $Q\hat{x},\hat{\lambda}\in \ker{R}$ and $Q\hat x + \hat\la = 0$.
\item[(ii)]  All $\hat{u}\in\mathbb U$ satisfying $(J-R)^{-1}B\hat{u} \in \ker{R}$ are optimal controls for \eqref{OCP:steady}.
\item[(iii)]  OCP \eqref{OCP:terminal} is strictly dissipative with respect to $\ker{R^{\frac12}Q}$ with storage function $H(x)$.
\item[(iv)]  If $R$ is invertible, then the unique optimal steady state is $\hat{x}=\hat{u}=0$ and OCP \eqref{OCP:terminal} is strictly dissipative with storage function $H(x)$.
\item[(v)] The available storage is given by $\tilde{S}(x) = \tfrac{1}{2} x^\top Q x$.
\end{enumerate}
\end{theorem}
\begin{proof}
Part (i) follows immediately from Theorem \ref{t:oss}. For (ii) we compute with Theorem~\ref{t:oss} for all optimal steady states that
\begin{align}
\label{eq:steady_cost_identity}
\hat{u}^\top \hat{y} = \|R^{\frac12}Q\hat{x}\|_2^2 = 0
\end{align}
as they are particular solutions of the pH~system with constant energy, i.e., choosing the (constant) control $\hat{u}$ and the initial state $\hat{x}$. For (iii) we use \eqref{eq:dissipativity} and obtain
\begin{align*}
H(x(T))-H(x(0)) = \int_0^T   u(t)^\top y(t) - \|R^{\frac12}Qx(t)\|^2\,\mathrm{d}t.
\end{align*}
To show Part (iv), we insert \eqref{eq:steady_cost_identity} into \eqref{OCP:steady}. As $R$ is positive definite, $R^{\frac12}Q$ is and we can estimate
\begin{align}
\label{eq:R_coercive}
\|R^{\frac12}Qv\| \geq \gamma \|v\|,
\end{align}
for $\gamma > 0$ and all $v\in \mathbb{R}^n$. Hence, by invertibility of  $(J-R)Q$, $\hat{u}=\hat{x}=0$ is the unique optimal solution with objective value zero. Moreover, \eqref{eq:R_coercive} yields strict dissipativity. Part (v) follows directly from $Q>0$ and Remark~\ref{rem:availStor}.
\end{proof}

\begin{remark}[Terminal cost instead of terminal value]
We briefly discuss the previous results in the case of a terminal cost, i.e., when replacing the terminal condition $x(T)=x_T$ in \eqref{OCP:terminal} by a terminal cost, i.e., minimizing $C_\varphi(u)\doteq C(u) + \varphi(x(T))$, where $\varphi:\mathbb{R}^n\to \mathbb{R}$ is a continuously differentiable Mayer term. The dissipativity notion introduced in Definition~\ref{def:DI_V} can be completely analogously defined for this case. In particular, Assumption (b) of Lemma~\ref{l:turnpike} is not necessary to conclude a turnpike result, as all terminal values $x(T)$ are feasible. Hence also Theorem~\ref{t:ph_turnpike} holds for the case of terminal cost when dropping the assumption that $x_T$ is reachable from zero.
\end{remark}

\begin{remark}[Connection with optimal steady states]
Theorem~\ref{thm:dissipativity} (i) states that optimal steady states lie in the conservative subspace, i.e., $\hat{x}\in \ker{R^{\frac12} Q}$. By the dissipativity (iii), we can conclude a turnpike property in the sense of \eqref{eq:turnpike_V} towards this subspace. This means, that solutions of the dynamic problem are close to the solutions of the steady state problem up to directions that lie in the conservative subspace. If $R$ is invertible, then this states a classical turnpike property towards the unique optimal steady state $\hat{x}=0$ by (iv).
\end{remark}
	
\begin{remark}[Regularization of the OCP]
	If we augment the cost functional with an additional control cost of the form $\int_0^T  \varepsilon \|u(t)\|^2\,\mathrm{d}t$, $\varepsilon >0$, by \eqref{eq:dissipativity} we obtain
	\begin{align*}
	&\int_0^T u(t)^\top y(t) + \varepsilon\|u(t)\|^2 \,\mathrm{d}t \\&= H(x(T))-H(x(0)) + \int_0^T  \|R^{\frac12}Qx(t)\|^2 + \varepsilon\|u(t)\|^2\,\mathrm{d}t.
	\end{align*}
	Then, assuming we have no specified terminal state, the optimization reduces to
	\begin{align*}
	\min_{u\in L^1(0,T;\mathbb{U})} &\int_0^T  \|R^{\frac12}Qx(t)\|^2 + \varepsilon\|u(t)\|^2\,\mathrm{d}t + H(x(T))\\
\text{s.t.}\quad	\dot{x}(t)&=(J-R)Qx + Bu\\
	x(0)&=x^0.
	\end{align*}
	In \cite[Proposition 1]{Pighin2020a} it was proven that if $((J-R)Q,B)$ is stabilizable one obtains the estimate
	\begin{align*}
	\|D(x(t)-\hat{x})\| + \|u(t)-\hat{u}\| \leq C(e^{-\mu t} + e^{-\mu(T-t)}),
	\end{align*}
	where $D\in \mathbb{R}^{n\times n}$ is a projection onto the detectable subspace, i.e., the observable subspace that corresponds to eigenvalues of $(J-R)Q$ with nonnegative real part. If $((J-R)Q,R^{\frac12}Q)$ is detectable, then $D=I$. Note, that this differs from our setting as we do not have a control penalization and control constraints, which rules out the Riccati theory used in \cite{Pighin2020a}. We briefly discuss a possible extension of Theorem~\ref{thm:dissipativity} to this control-regularized case. One immediately sees that the unique optimal steady state is given by $\hat{x}=\hat{u}=0$ with associated Lagrange multiplier $\hat{\lambda}=0$. Hence claim (i) of Theorem~\ref{thm:dissipativity} trivially holds. It is clear, that also the dissipation inequality \eqref{eq:sDI} still holds as we only add the positive term $\int_0^T\varepsilon \|u(t)\|^2\,dt$ on the right-hand side. Hence the claims (iii) and (iv) of Theorem~\ref{thm:dissipativity} also remain valid. The second claim (ii) does obviously not hold by uniqueness of the optimal control. 
\end{remark}

\section{NUMERICAL EXAMPLE}\label{sec:example}
\noindent We consider a mass-spring damper with external force as sketched in Figure~\ref{fig:ex1}.
\begin{figure}[!h]
	\centering
	\includegraphics[scale=0.7]{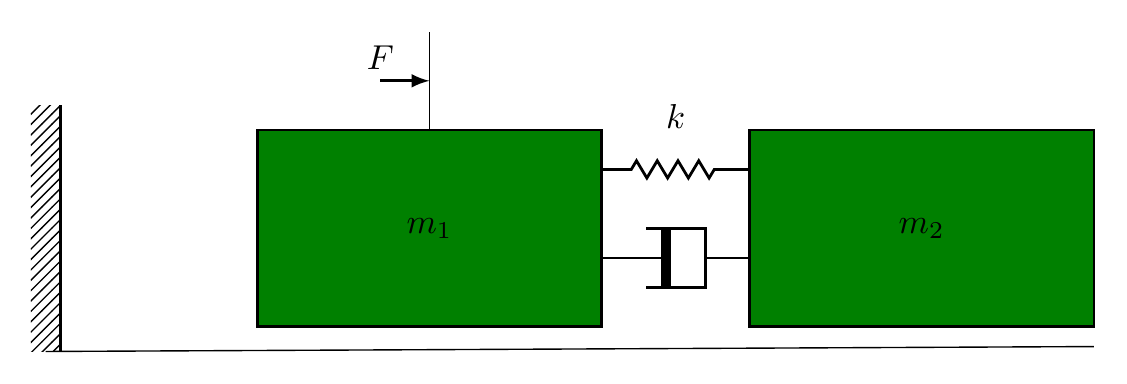}
	\caption{Mass-spring damper with external force.}\label{fig:ex1}
\end{figure}
The vector of energy variables is given by $x=\left(
	p_{1},
	p_{2},
	q
\right)$ with momenta $p_{1},p_2$ and displacement of the spring $q$. The total mechanical energy is defined as $H\left(x\right)=\frac{p_{1}^{2}}{2\,m_{1}}+\frac{p_{2}^{2}}{2\,m_{1}}+\frac{1}{2}k\,q^{2}$. This leads to $Q \doteq \operatorname{diag} ( 1/m_{1},1/m_{2},k)$, which we normalize to be the identity.

We consider a friction force (setting later the friction coefficient $\nu=1$)
$$
	F_{r}\doteq\nu\left(\frac{dq}{dt}\right)=\nu\left(\frac{p_{2}}{m_{2}}-\frac{p_{1}}{m_{1}}\right).
$$
Hence, the Poisson and the dissipation matrix 
are given by
\begin{align}\label{eq:exR1}
	J \doteq \begin{pmatrix} 
	\phantom{-}0 & 0 & \phantom{-}1\\
	\phantom{-}0 & 0 & -1\\
	-1 & 1 & \phantom{-}0
	\end{pmatrix}, 
	\quad R \doteq \begin{pmatrix} 
		\phantom{-}1&-1&0\\
		-1&\phantom{-}1&0\\
		\phantom{-}0&\phantom{-}0&0
	\end{pmatrix}. 
\end{align}
The input matrix is $B=\begin{pmatrix}
1&0&0
\end{pmatrix}^\top$.

\subsection{One dimensional subspace turnpike}

\noindent Let the initial and terminal state by given by $x^0=(1,1,1)^\top$ and $x_T = (-1.2,-0.7,-1)^\top$, respectively. We solve the corresponding OCP \eqref{OCP:terminal} with the 
\textsc{MATLAB} toolbox \emph{fmincon}. 

Figure~\ref{fig:01} depicts pairs of optimal controls and trajectories 
for three different time horizons $T \in \{10,30,50\}$. We observe that no variable exhibits a classical turnpike in the sense that state or control approach a steady state. 

\begin{figure}[!htb]
	\begin{center}
		\includegraphics[scale=.6]{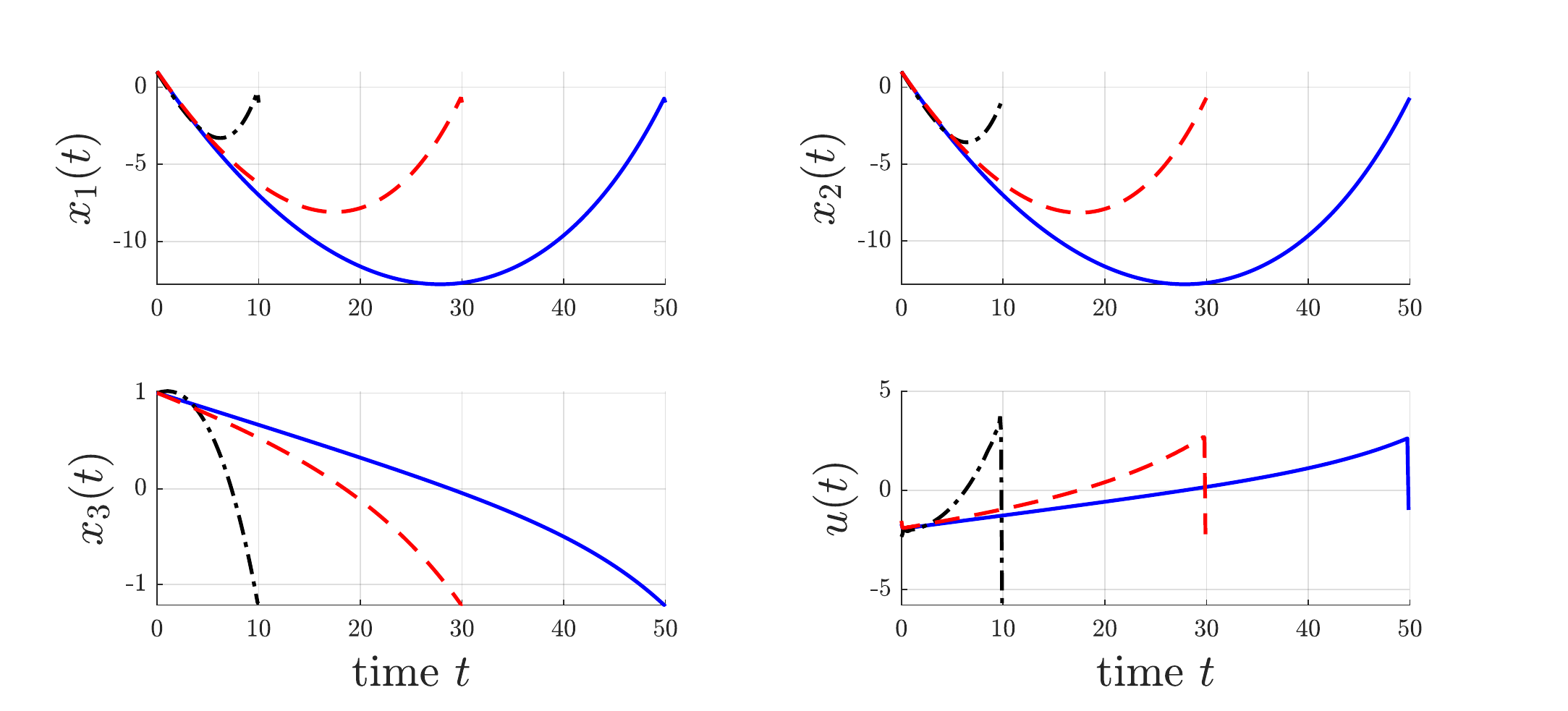}
	\end{center}
	\caption{Optimal state and control of OCP~\eqref{OCP:terminal} for the mass-spring-damper system~\eqref{eq:exR1} for time horizons $T=10$ ($- \cdot$), $T=30$ (\textcolor{red}{$--$}), $T=50$  (\textcolor{blue}{---}). \label{fig:01}}
\end{figure}

\begin{figure}[!htb]
	\begin{center}
				\includegraphics[scale=.5]{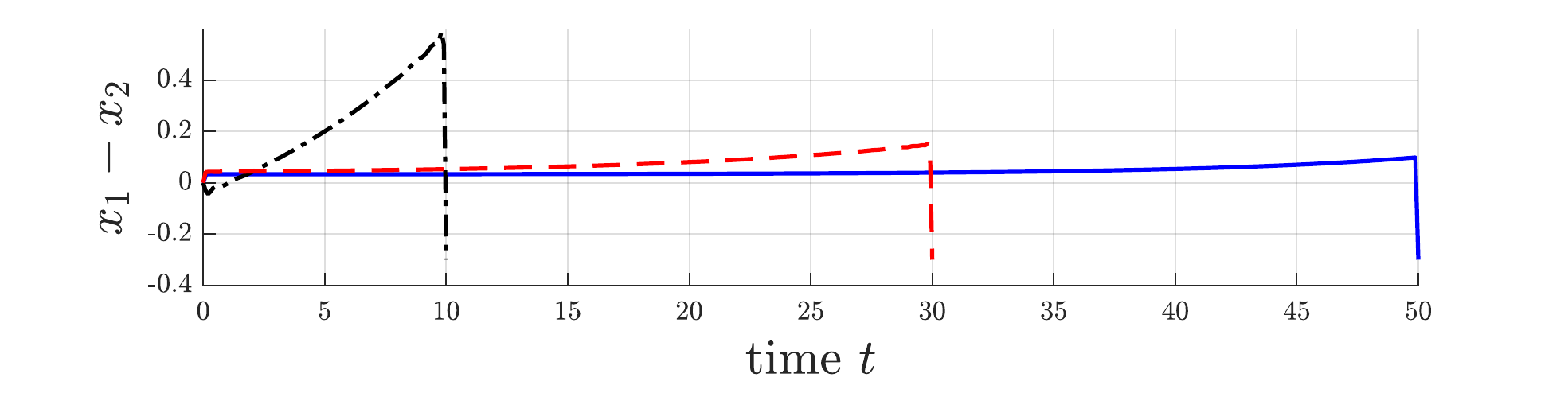}\\
		\includegraphics[scale=.5]{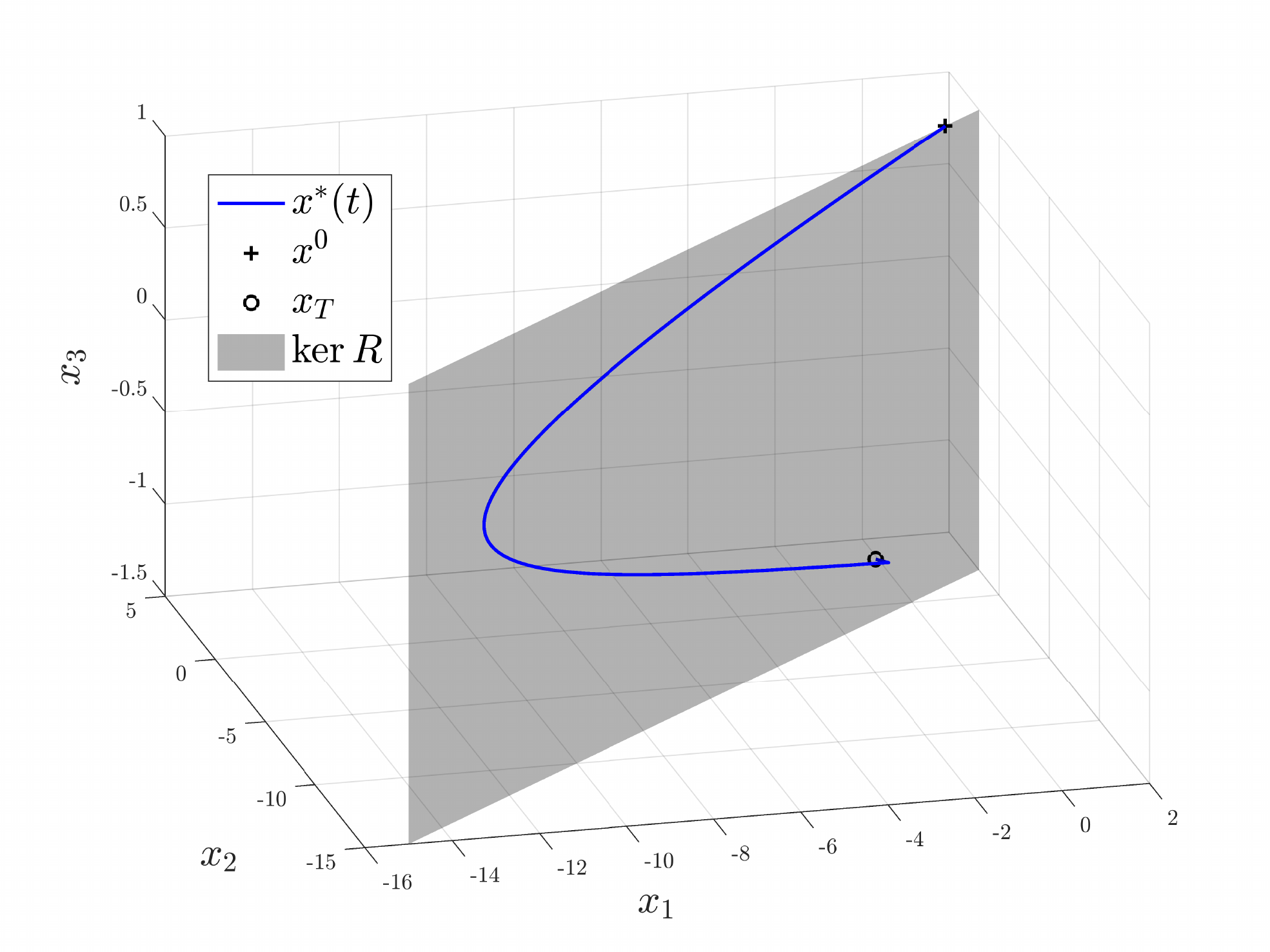}
	\end{center}
	\caption{Optimal state trajectories of OCP~\eqref{OCP:terminal} for the mass-spring-damper system~\eqref{eq:exR1}: illustration of the distance $|x_1-x_2|$ to the non-dissipative subspace~\eqref{e:kern1} for horizons $T=10$ ($- \cdot$), $T=30$ (\textcolor{red}{$--$}), $T=50$  (\textcolor{blue}{---}).
	\label{fig:02}}
\end{figure}

However, if we additionally indicate the non-dissipative subspace \begin{align}
\label{e:kern1}
\ker R^{\frac12}Q= \ker R=\{x\in \mathbb{R}^3\,|\,x_1-x_2=0\},
\end{align}it can be clearly seen that the state approaches this subspace as proven in Theorem \ref{t:ph_turnpike}, see Figure~\ref{fig:02}.


\subsection{Two dimensional subspace turnpike}

\noindent Next, we slightly modify the dissipation matrix to illustrate the case where the subspace $N_1$ of Lemma~\ref{l:dec} is two-dimensional. To this end, we set
\begin{align}
	J\doteq \begin{pmatrix} 
		\phantom{-}0 & 0 & \phantom{-}1\\
		\phantom{-}0 & 0 & -1\\
		-1 & 1 & \phantom{-}0
	\end{pmatrix}, 
	\quad R\doteq \begin{pmatrix} 
		1&1&0\\1&1&0\\0&0&0
	\end{pmatrix}. 
	\label{eq:exR2}
\end{align}
In Figure~\ref{fig:1} we depict the optimal state and control 
for the three different time horizons $T \in \{10,15,20\}$. 
\begin{figure}[H]
	\begin{center}
	\includegraphics[scale=.6]{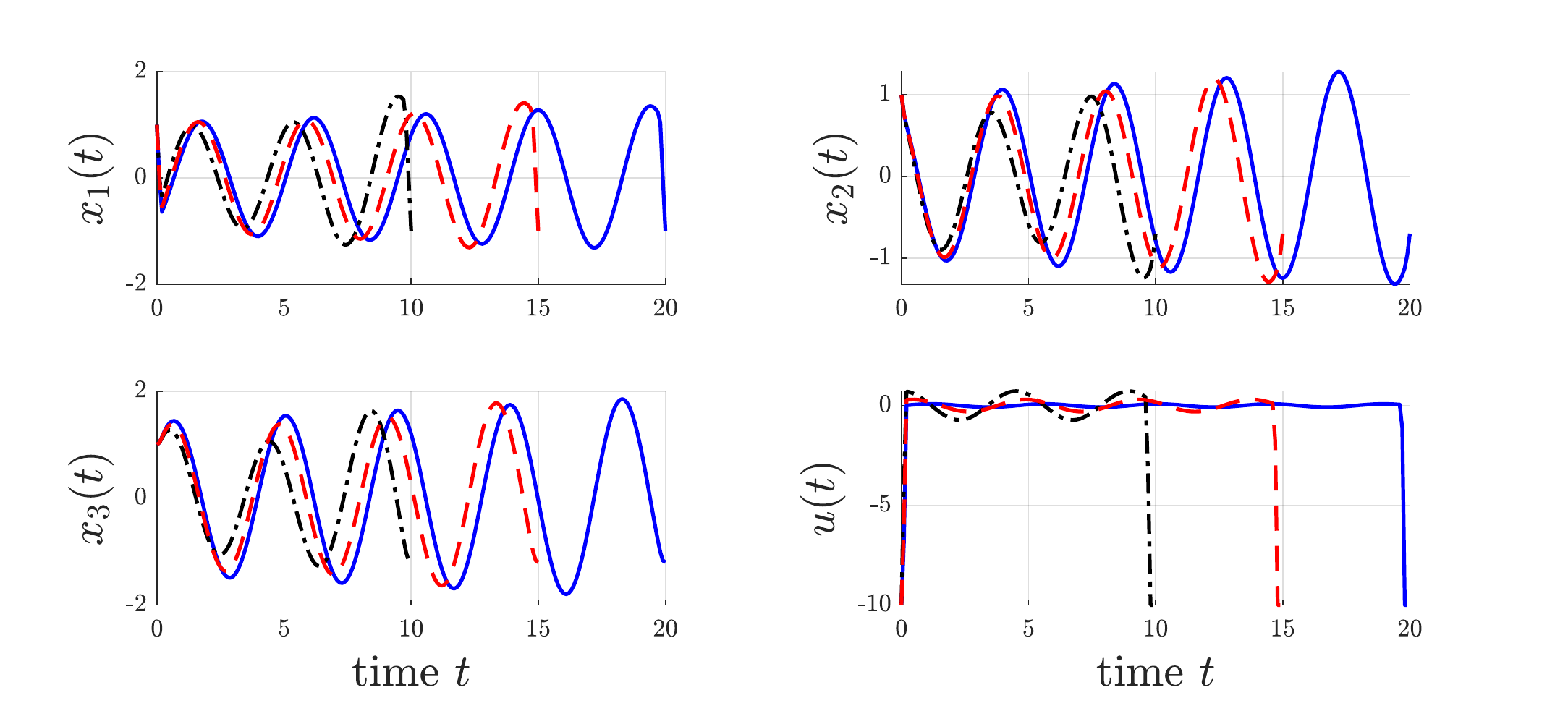}
	\end{center}
\caption{Optimal solution of OCP~\eqref{OCP:terminal} for the modified mass-spring-damper system~\eqref{eq:exR2} for time horizons $T=10$ ($- \cdot$), $T=15$ (\textcolor{red}{$--$}), $T=20$  (\textcolor{blue}{---}).\label{fig:1}}
\end{figure}

We can clearly see that all states are fully dynamic over the entire horizon, i.e., they do not approach a classical steady state turnpike. The control, however, is close to zero for the majority of the time.
The subspace turnpike phenomenon proved in Theorem~\ref{t:ph_turnpike} can be observed in Figures~\ref{fig:2}: the optimal state approaches the subspace 
\begin{align}
\label{e:kern2}
\ker R^{\frac12}Q = \ker R = \{x\in \mathbb{R}^3\,|\,x_1+x_2=0\}
\end{align} and the behavior is dominated by the skew symmetric matrix $J_1$ corresponding to the decomposition in Lemma~\ref{l:dec}.
\begin{figure}[htb]
	\begin{center}
		\includegraphics[scale=.6]{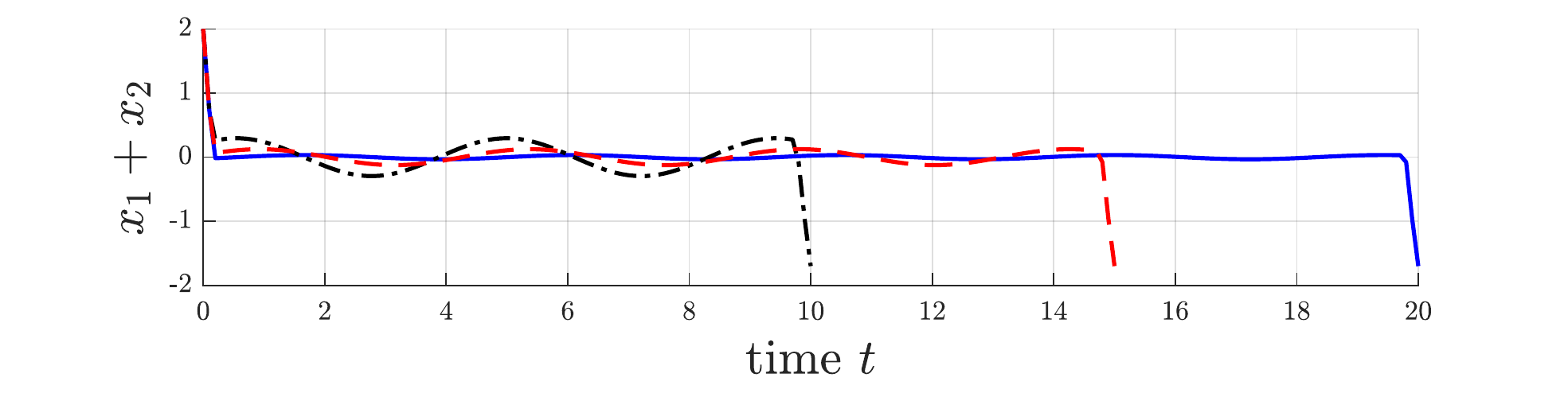} \\
		\includegraphics[scale=.6]{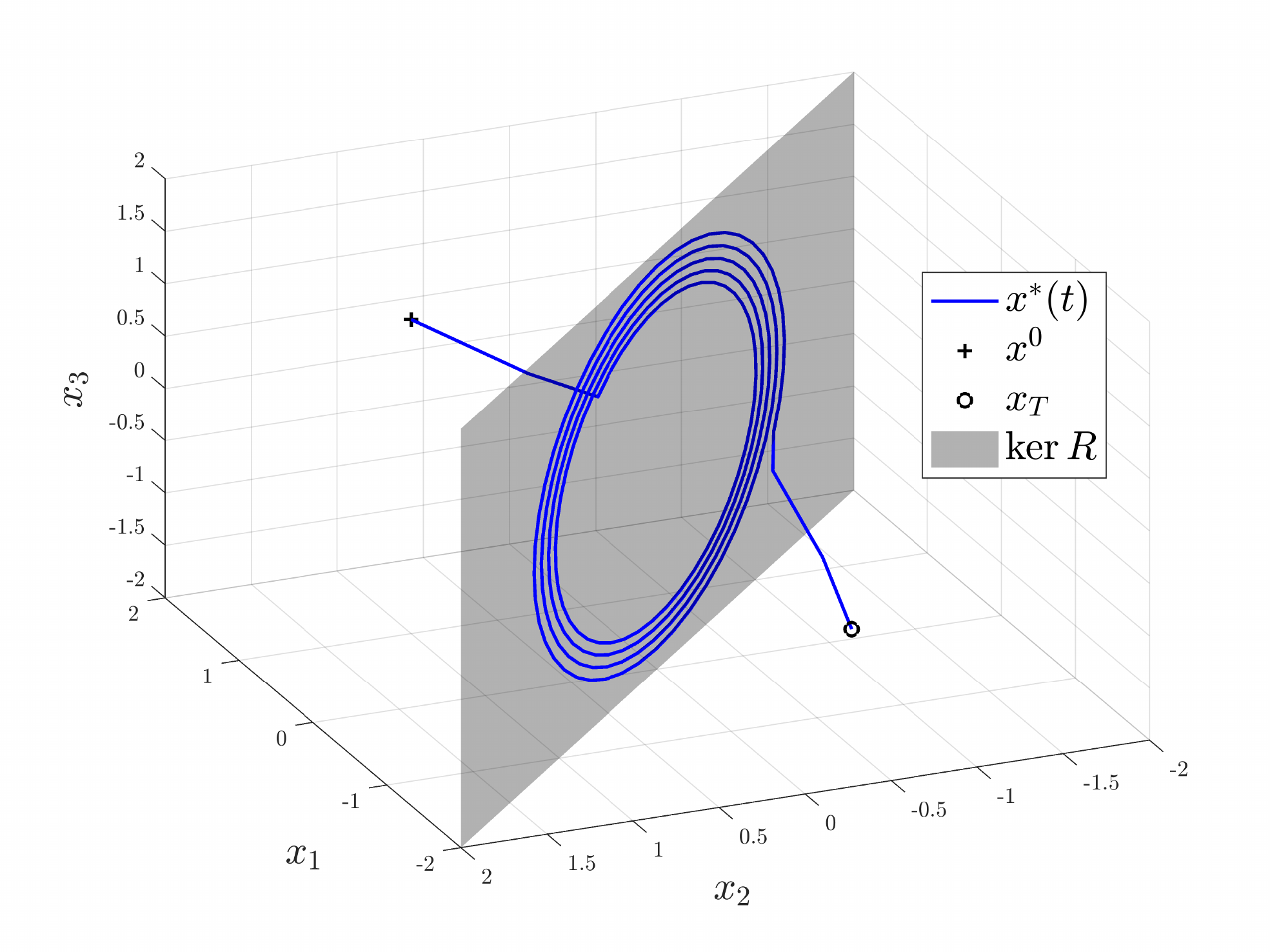}
	\end{center}
	\caption{Optimal state trajectories of OCP~\eqref{OCP:terminal} for the modified mass-spring-damper system~\eqref{eq:exR2}: illustration of the distance $|x_1+x_2|$ to the non-dissipative subspace \eqref{e:kern2} for horizons $T=10$ ($- \cdot$), $T=15$ (\textcolor{red}{$--$}), $T=20$  (\textcolor{blue}{---}).}\label{fig:2}
\end{figure}

\FloatBarrier

\section{Conclusions}
\noindent This paper has studied optimal control problems for linear port-Hamiltonian systems. Specifically, we consider the problem of state transition 
while minimizing the intrinsic pH~objective, i.e. the supplied energy. We have shown that under mild assumptions the considered OCPs are strictly dissipative w.r.t.\ 
the kernel of the energy-dissipation matrix~$RQ$, i.e., the structure matrix of the generalized gradient part of the system. 
This induces the turnpike phenomenon w.r.t.\ a subsapce, i.e., w.r.t. the kernel of the gradient structure matrix. 
Finally, we have drawn upon a numerical example to illustrate the interplay between the energy-dissipation matrix and the structure of the turnpike in the optimal solutions.

\bibliographystyle{abbrv}
\bibliography{references}   

\begin{thebibliography}{10}

\bibitem{Angeli12a}
D.~{Angeli}, R.~{Amrit}, and J.~B. {Rawlings}.
\newblock On average performance and stability of economic model predictive
  control.
\newblock {\em IEEE Transactions on Automatic Control}, 57(7):1615--1626, 2012.

\bibitem{beattie2018robust}
C.~A. Beattie, V.~Mehrmann, and P.~Van~Dooren.
\newblock Robust port-{H}amiltonian representations of passive systems.
\newblock {\em Automatica}, 100:182--186, 2019.

\bibitem{BrogliatoSpringer20}
B.~Brogliato, R.~Lozano, B.~Maschke, and O.~Egeland.
\newblock {\em Dissipative {S}ystems {A}nalysis and {C}ontrol}.
\newblock Communications and Control Engineering Series. Springer Cham, 3rd
  edition, 2020.

\bibitem{Carlson91}
D.~Carlson, A.~Haurie, and A.~Leizarowitz.
\newblock {\em Infinite Horizon Optimal Control: Deterministic and Stochastic
  Systems}.
\newblock Springer Verlag, 1991.

\bibitem{Stieler14a}
T.~Damm, L.~Gr{\"u}ne, M.~Stieler, and K.~Worthmann.
\newblock An exponential turnpike theorem for dissipative optimal control
  problems.
\newblock {\em SIAM Journal on Control and Optimization}, 52(3):1935--1957,
  2014.

\bibitem{Dorfman58}
R.~Dorfman, P.~Samuelson, and R.~Solow.
\newblock {\em Linear Programming and Economic Analysis}.
\newblock McGraw-Hill, New York, 1958.

\bibitem{tudo:faulwasser21b}
T.~Faulwasser and L.~Gr\"une.
\newblock {\em Turnpike Properties in Optimal Control: An Overview of
  Discrete-Time and Continuous-Time Results}.
\newblock Elsevier, 2021.
\newblock arxiv: 2011.13670. In press.

\bibitem{epfl:faulwasser15h}
T.~Faulwasser, M.~Korda, C.~Jones, and D.~Bonvin.
\newblock On turnpike and dissipativity properties of continuous-time optimal
  control problems.
\newblock {\em Automatica}, 81:297--304, 2017.

\bibitem{Fuller60a}
A.~Fuller.
\newblock Relay control systems optimized for various performance criteria.
\newblock In {\em Automatic and remote control, Proc. first IFAC world
  congress}, volume~1, pages 510--519, 1960.

\bibitem{Gruene2019a}
L.~Gr{\"u}ne and R.~Guglielmi.
\newblock On the relation between turnpike properties and dissipativity for
  continuous time linear quadratic optimal control problems.
\newblock {\em Mathematical Control and Related Fields}, Online First, 2020.

\bibitem{Gruene2016a}
L.~Gr\"une and M.~A. M\"uller.
\newblock On the relation between strict dissipativity and turnpike properties.
\newblock {\em Systems \& Control Letters}, 90:45--53, 2016.

\bibitem{Gruene2019}
L.~Gr\"une, M.~Schaller, and A.~Schiela.
\newblock Exponential sensitivity and turnpike analysis for linear quadratic
  optimal control of general evolution equations.
\newblock {\em Journal of Differential Equations}, 268(12):7311--7341, 2020.

\bibitem{Jacob2012}
B.~Jacob and H.~J. Zwart.
\newblock {\em Linear port-{H}amiltonian systems on infinite-dimensional
  spaces}, volume 223.
\newblock Springer Science \& Business Media, 2012.

\bibitem{Koelsch2020}
L.~K{\"o}lsch, P.~J. Soneira, F.~Strehle, and S.~Hohmann.
\newblock Optimal control of {P}ort-{H}amiltonian systems: A time-continuous
  learning approach, 2020.
\newblock arXiv:2007.08645.

\bibitem{Lamoline_MTNS18_LQGStochPHS}
F.~Lamoline and J.~J. Winkin.
\newblock On {LQG} control of stochastic port-{H}amiltonian systems on
  infinite-dimensional spaces.
\newblock In {\em 23rd International Symposium on Mathematical Theory of
  Networks and Systems}, pages 197--203, 2018.

\bibitem{Liberzon12}
D.~Liberzon.
\newblock {\em Calculus of Variations and Optimal Control Theory: A Concise
  Introduction}.
\newblock Princeton University Press, 2012.

\bibitem{Macki2012}
J.~Macki and A.~Strauss.
\newblock {\em Introduction to optimal control theory}.
\newblock Springer Science \& Business Media, 2012.

\bibitem{Mckenzie76}
L.~McKenzie.
\newblock Turnpike theory.
\newblock {\em Econometrica: Journal of the Econometric Society},
  44(5):841--865, 1976.

\bibitem{Mehrmann20}
V.~Mehrmann and P.~M. Van~Dooren.
\newblock Optimal robustness of port-{H}amiltonian systems.
\newblock {\em SIAM Journal on Matrix Analysis and Applications},
  41(1):134--151, 2020.

\bibitem{Moylan14a}
P.~Moylan.
\newblock Dissipative systems and stability.
\newblock {\em Lecture notes, University of Newcastle, www.pmoylan.org}, 2014.

\bibitem{Ortega08}
R.~Ortega, A.~Van Der~Schaft, F.~Castanos, and A.~Astolfi.
\newblock Control by interconnection and standard passivity-based control of
  port-{H}amiltonian systems.
\newblock {\em IEEE Transactions on Automatic Control}, 53(11):2527--2542,
  2008.

\bibitem{Pighin2020a}
D.~Pighin and N.~Sakamoto.
\newblock The turnpike with lack of observability, 2020.
\newblock arXiv:2007.14081.

\bibitem{Sato17}
K.~Sato.
\newblock Riemannian optimal control and model matching of linear
  port-{H}amiltonian systems.
\newblock {\em IEEE Transactions on Automatic Control}, 62(12):6575--6581,
  2017.

\bibitem{Sepulchre97a}
R.~Sepulchre, M.~Jankovic, and P.~Kokotovic.
\newblock {\em Constructive Nonlinear Control}.
\newblock Springer Science \& Business Media, 1st edition, 1997.

\bibitem{Sussmann97}
H.~Sussmann and J.~Willems.
\newblock 300 years of optimal control: from the brachystochrone to the maximum
  principle.
\newblock {\em IEEE Control Systems}, 17(3):32--44, 1997.

\bibitem{Trelat2020}
E.~Tr\'elat.
\newblock Linear turnpike theorem, 2020.
\newblock arXiv:2010.13605.

\bibitem{Trelat15a}
E.~Tr{\'e}lat and E.~Zuazua.
\newblock The turnpike property in finite-dimensional nonlinear optimal
  control.
\newblock {\em Journal of Differential Equations}, 258(1):81--114, 2015.

\bibitem{Troeltzsch2010}
F.~Tr{\"o}ltzsch.
\newblock {\em Optimal control of partial differential equations: theory,
  methods, and applications}.
\newblock American Mathematical Soc., 2010.

\bibitem{van2008balancing}
A.~Van Der~Schaft.
\newblock Balancing of lossless and passive systems.
\newblock {\em IEEE Transactions on Automatic Control}, 53(9):2153--2157, 2008.

\bibitem{van2014port}
A.~Van Der~Schaft and D.~Jeltsema.
\newblock Port-{H}amiltonian systems theory: An introductory overview.
\newblock {\em Foundations and Trends in Systems and Control}, 1(2-3):173--378,
  2014.

\bibitem{Villanueva2020}
M.~E. Villanueva, E.~D. Lazzari, M.~A. M\"uller, and B.~Houska.
\newblock A set-theoretic generalization of dissipativity with applications in
  tube {MPC}.
\newblock {\em Automatica}, 122:109179, 2020.

\bibitem{Willems1972a}
J.~C. Willems.
\newblock Dissipative dynamical systems part i: General theory.
\newblock {\em Archive for rational mechanics and analysis}, 45(5):321--351,
  1972.

\bibitem{Automatica18_Wu}
Y.~Wu, B.~Hamroun, Y.~Le~Gorrec, and B.~Maschke.
\newblock Reduced order {LQG} control design for port {H}amiltonian systems.
\newblock {\em Automatica}, 95:86--92, 2018.

\end{thebibliography}
\end{document}